\documentclass[a4paper,12pt,notitlepage]{amsart}

\usepackage{amsfonts}
\usepackage{amsmath}
\usepackage{amssymb}
\usepackage{enumerate}
\usepackage{hyperref}
\usepackage{mathrsfs}

\newtheorem{theorem}{Theorem}
\newtheorem{corollary}{Corollary}

\newtheorem{definition}{Definition}

\newtheorem{lemma}{Lemma}
\newtheorem{remark}{Remark}
\numberwithin{equation}{section}
\input{tcilatex}

\begin{document}
\title[Chern-Simons forms for $\mathbb{R}$-linear connections]{Chern-Simons
forms for $\mathbb{R}$-linear connections on\\
Lie algebroids}
\author{Bogdan Balcerzak}
\date{}
\maketitle

\begin{abstract}
The Chern-Simons forms for $\mathbb{R}$-linear connections on Lie algebroids
are considered. A generalized Chern-Simons formula for such $\mathbb{R}$%
-linear connections is obtained. We it apply to define Chern character and
secondary characteristic classes for $\mathbb{R}$-linear connections of Lie
algebroids.
\end{abstract}

\renewcommand{\thefootnote}{}\footnote{\hspace{-0.6cm}2010 \emph{Mathematics
Subject Classification}: Primary 53C05; Secondary 58H05, 17B56.}\footnote{%
\hspace{-0.6cm}\emph{Key words and phrases}: Lie algebroid, connections, Lie
algebroid cohomology, Chern-Simons forms.}

\section{Introduction}

We observe that non-linear objects (forms, connections, mappings between
modules of cross-sections of vector bundles, which are non--linear over a
ring of smooth functions) have increasing meaning in problems of
differential geometry. S.~Evens, J.~H.~Lu and A.~Weinstein considered
especial non--linear connections of Lie algebroids called connections up to
homotopy (see \cite{Evens-Lu-Weinstein}). Crainic and Fernandes \cite%
{Crainic-up to homotopy}, \cite{Crainic-Fernandes-jets} introduce the Chern
character for non--linear connections. They discuss non--linear forms on Lie
algebroids with values in a super vector bundle as antisymmetric,
multilinear maps over $\mathbb{R}$ (not necessarily multilinear over the
ring of smooth functions), which have a local property. Every non--linear
connection $\nabla $ establishes on non--linear forms the covariant
derivative operator. If $\nabla $ is flat, the Chern character vanishes and
the induced covariant derivative operator is the exterior derivative, and in
classically way defines the cohomology space. Crainic and Fernandes
introduced secondary characteristic classes for connections up to homotopy 
\cite{Crainic-up to homotopy}, \cite{Crainic-Fernandes-jets}. We stay the
question whenever these ideas refer to $\mathbb{R}$-linear forms and $%
\mathbb{R}$-linear connections -- meaning as objects for which it is not
supposed a local property. In the paper, using the generalized Stokes
formula for $\mathbb{R}$-linear connections on Lie algebroids, we prove the
Chern-Simons transgression formula without assumption locality for $\mathbb{R%
}$-linear connections. This is a helpful starting point to define
characteristic classes for $\mathbb{R}$-linear connections on Lie
algebroids. Some Crainic and Fernandes ideas we use to extend notions of
Chern character and exotic (secondary) characteristic classes to $\mathbb{R}$%
-linear objects. Moreover, we found some explicit formulae for $\mathbb{R}$%
-linear Chern-Simons forms. In particular, we gain an direct formula of
exotic (secondary) characteristic classes for an $\mathbb{R}$-linear
connection as some trace $\mathbb{R}$-linear forms on a Lie algebroid.

A \emph{Lie algebroid} is a trip $\left( A,\rho _{A},[\![\bullet ,\bullet
]\!]_{A}\right) $, in which $A$ is a real vector bundle over a manifold $M$, 
$\rho _{A}:A\rightarrow TM$ (called an \emph{anchor}) is a homomorphism of
vector bundles, $\left( \Gamma \left( A\right) ,[\![\bullet ,\bullet
]\!]_{A}\right) $ is an $\mathbb{R}$-Lie algebra and the Leibniz identity%
\begin{equation*}
\lbrack \![a,f\cdot b]\!]_{A}=f\cdot \lbrack \![a,b]\!]_{A}+\rho _{A}\left(
a\right) \left( f\right) \cdot b\ \ \ \ \ \text{for all\ \ \ \ }a,b\in
\Gamma \left( A\right) ,\ f\in 
%TCIMACRO{\TeXButton{C}{\mathscr{C}}}%
%BeginExpansion
\mathscr{C}%
%EndExpansion
^{\infty }\left( M\right)
\end{equation*}%
holds. Since the representation $\varrho :%
%TCIMACRO{\TeXButton{C}{\mathscr{C}}}%
%BeginExpansion
\mathscr{C}%
%EndExpansion
^{\infty }\left( M\right) \rightarrow \limfunc{End}\nolimits_{%
%TCIMACRO{\TeXButton{C}{\mathscr{C}}}%
%BeginExpansion
\mathscr{C}%
%EndExpansion
^{\infty }\left( M\right) }\left( \Gamma \left( A\right) \right) $, $\varrho
\left( \nu \right) \left( a\right) =\nu \cdot a$, $\nu \in 
%TCIMACRO{\TeXButton{C}{\mathscr{C}}}%
%BeginExpansion
\mathscr{C}%
%EndExpansion
^{\infty }\left( M\right) $,$\ a\in \Gamma \left( A\right) $, is faithful (%
\cite{Herz}, see also \cite{B-K-W-Primary}), the anchor induces a
homomorphism of Lie algebras $\limfunc{Sec}\rho _{A}:\Gamma \left( A\right)
\rightarrow 
%TCIMACRO{\TeXButton{X}{\mathscr{X}}}%
%BeginExpansion
\mathscr{X}%
%EndExpansion
\left( M\right) $, $a\mapsto \rho _{A}\circ a$. If $\rho _{A}$ is a constant
rank (i.e. $\func{Im}\rho _{A}$ is a constant dimensional and completely
integrable distribution), we say that $\left( A,\rho _{A},[\![\bullet
,\bullet ]\!]_{A}\right) $ is \emph{regular}. A tangent bundle $TM$ to a
manifold $M$ with the identity as an anchor and the bracket of vector fields
is an elementary example of a Lie algebroid. For more about Lie algebroids
and their properties we refer for example to \cite{Mackenzie}, \cite%
{Higgins-Mackenzie}, \cite{Kubarski-Lyon}, \cite{Fernandes}, \cite%
{B-K-W-Primary}, \cite{Crainic-Fernandes-jets}.

There are Lie functors from many geometric categories to the category of Lie
algebroids (see a long list eg in \cite{Mackenzie}, \cite{Kubarski-Lyon}).
Especially meaning in the paper have algebroids of vector bundles. We recall
that the module $%
%TCIMACRO{\TeXButton{CDO}{\mathscr{CDO}}}%
%BeginExpansion
\mathscr{CDO}%
%EndExpansion
\left( E\right) $ of sections of the Lie algebroid $\limfunc{A}\left(
E\right) $ of a vector bundle $E$ is the space of all covariant differential
operators in $E$, i.e. $\mathbb{R}$-linear operators $\ell :\Gamma \left(
E\right) \rightarrow \Gamma \left( E\right) $ such that there exists exactly
one $\widetilde{\ell }\in 
%TCIMACRO{\TeXButton{X}{\mathscr{X}}}%
%BeginExpansion
\mathscr{X}%
%EndExpansion
\left( M\right) $ with $\ell \left( f\zeta \right) =f\ell \left( \zeta
\right) +\widetilde{\ell }\left( f\right) \zeta $ for all $f\in 
%TCIMACRO{\TeXButton{C}{\mathscr{C}}}%
%BeginExpansion
\mathscr{C}%
%EndExpansion
^{\infty }\left( M\right) $ and $\zeta \in \Gamma \left( E\right) $; see for
example \cite{Teleman}, \cite{Mackenzie}, \cite{Kubarski-Lyon}.

Let $\left( A,\rho _{A},[\![\bullet ,\bullet ]\!]_{A}\right) $\ and $\left(
B,\rho _{B},[\![\bullet ,\bullet ]\!]_{B}\right) $\ be Lie algebroids over
the same manifold $M$. A homomorphism $\nabla :A\rightarrow B$ of vector
bundles is called an $A$-\emph{connection} in $B$ if $\rho _{B}\circ \nabla
=\rho _{A}$ (see \cite{B-K-W-Primary}). If an $A$-connection $\nabla $ in $B$
is a homomorphism of Lie algebroids ($\nabla $ preserves the Lie brackets)
we say that $\nabla $ is \emph{flat}. The notion of an $A$-connection in $B$
generalizes the known notions of connections (for example usual and partial
covariant derivatives in vector bundles, a connection in principal bundles,
a connection in extensions of Lie algebroids). In the case where $A=TM$ and $%
B=\limfunc{A}\left( E\right) $ is an algebroid of a vector bundle $E$, $TM$%
-connections in $\limfunc{A}\left( E\right) $ are one--to--one with
covariant derivatives in $E$. For an arbitrary Lie algebroid $A$ and $B=%
\limfunc{A}\left( E\right) $ we have $A$-connections of $E$ considered in 
\cite{Mackenzie}, \cite{Fernandes}, \cite{Crainic-Fernandes-jets}. In case $%
B=\limfunc{A}\left( P\right) $ is a Lie algebroid of a principal bundle $P$,
we get $A$-connections in $P$. In Poisson geometry an especially rule have
connections acting from a Lie algebroid $T^{\ast }M$ associated to a given
Poisson structure. In these examples a connection $\nabla $ considered as a
mapping on modules of cross-sections is linear over $%
%TCIMACRO{\TeXButton{C}{\mathscr{C}}}%
%BeginExpansion
\mathscr{C}%
%EndExpansion
^{\infty }\left( M\right) $.

By an $\mathbb{R}$\emph{-linear connection\ }of $A$\emph{\ }in $B$\emph{\ }%
we called an $\mathbb{R}$-linear operator $\nabla :\Gamma \left( A\right)
\rightarrow \Gamma \left( B\right) $\emph{\ }such that 
\begin{equation*}
\limfunc{Sec}\rho _{B}\circ \nabla =\limfunc{Sec}\rho _{A}\text{.}
\end{equation*}%
An $\mathbb{R}$-linear connection of $A$ in the Lie algebroid $A\left(
E\right) $ is called the $\mathbb{R}$\emph{-linear connection of }$A$\emph{\
on the vector bundle} $E$. We call the map%
\begin{equation*}
R^{\nabla }:\Gamma \left( A\right) \times \Gamma \left( A\right) \rightarrow
\Gamma \left( B\right) ,\ \ R^{\nabla }\left( \alpha ,\beta \right)
=[\![\nabla _{\alpha },\nabla _{\beta }]\!]_{B}-\nabla _{\lbrack \![\alpha
,\beta ]\!]_{A}}
\end{equation*}%
a \emph{curvature} of $\nabla $. We see that $\nabla :\Gamma \left( A\right)
\rightarrow \Gamma \left( B\right) $ is flat if $R^{\nabla }=0$. For every
Lie algebroid $A$, the adjoint connection $\limfunc{ad}:\Gamma \left(
A\right) \rightarrow 
%TCIMACRO{\TeXButton{CDO}{\mathscr{CDO}}}%
%BeginExpansion
\mathscr{CDO}%
%EndExpansion
\left( A\right) $, $\limfunc{ad}\left( a\right) =[\![a,\bullet ]\!]_{A}$ is
an $\mathbb{R}$-linear connection of $A$ on $A$. The notion of an $\mathbb{R}
$-linear connection includes so-called non-linear connections and
connections up to homotopy on super-vector bundles (\cite{Crainic-up to
homotopy}, \cite{Crainic-Fernandes-jets}, \cite{Evens-Lu-Weinstein}); such
connections have a local property.

Let $\left( A,\rho _{A},[\![\bullet ,\bullet ]\!]_{A}\right) $, $\left(
B,\rho _{B},[\![\bullet ,\bullet ]\!]_{B}\right) $\ be Lie algebroids over a
manifold $M$. An $\mathbb{R}$-multilinear, antisymmetric map 
\begin{equation*}
\omega :\underset{n}{\underbrace{\Gamma \left( A\right) \times \cdots \times
\Gamma \left( A\right) }}\longrightarrow \Gamma \left( B\right) 
\end{equation*}%
is called an $\mathbb{R}$\emph{-linear }$n$-\emph{form} on $A$ with values
in $B$. The space of all such $\mathbb{R}$-linear $n$-forms will be denoted
by $\mathcal{A}lt_{\mathbb{R}}^{n}\left( \Gamma \left( A\right) ;\Gamma
\left( B\right) \right) $, and the space of $\mathbb{R}$-linear forms on $A$
with values in $B$ by%
\begin{equation*}
\mathcal{A}lt_{\mathbb{R}}^{\bullet }\left( \Gamma \left( A\right) ;\Gamma
\left( B\right) \right) =\bigoplus\limits_{k\geq 0}\mathcal{A}lt_{\mathbb{R}%
}^{k}\left( \Gamma \left( A\right) ;\Gamma \left( B\right) \right) ,
\end{equation*}%
where $\mathcal{A}lt_{\mathbb{R}}^{0}\left( \Gamma \left( A\right) ;\Gamma
\left( B\right) \right) =\Gamma \left( B\right) $.{\huge \ }Observe that if $%
\nabla :A\rightarrow B$ is an arbitrary $\mathbb{R}$-linear connection, then
the curvature $R^{\nabla }$ is an element of $\mathcal{A}lt_{\mathbb{R}%
}^{2}\left( \Gamma \left( A\right) ;\Gamma \left( B\right) \right) $. We
define the covariant differential operator%
\begin{equation*}
d_{\mathbb{R}}^{\nabla }:\mathcal{A}lt_{\mathbb{R}}^{\bullet }\left( \Gamma
\left( A\right) ;\Gamma \left( B\right) \right) \longrightarrow \mathcal{A}%
lt_{\mathbb{R}}^{\bullet +1}\left( \Gamma \left( A\right) ;\Gamma \left(
B\right) \right) 
\end{equation*}%
for $\mathbb{R}$-linear forms on $A$ with values in $B$ by the classical
formula%
\begin{multline*}
\left( d_{\mathbb{R}}^{\nabla }\eta \right) \left( a_{1},\ldots
,a_{n+1}\right) =\dsum\limits_{i=1}^{n+1}\left( -1\right) ^{i+1}\nabla
_{a_{i}}\left( \eta \left( a_{1},\ldots \hat{\imath}\ldots ,a_{n+1}\right)
\right)  \\
+\dsum\limits_{i<j}\left( -1\right) ^{i+j}\eta \left( \lbrack
\![a_{i},a_{j}]\!]_{A},a_{1},\ldots \hat{\imath}\ldots \hat{\jmath}\ldots
,a_{n+1}\right) .
\end{multline*}%
$d_{\mathbb{R}}^{\nabla }$ is an antiderivation in $\mathcal{A}lt_{\mathbb{R}%
}^{\bullet }\left( \Gamma \left( A\right) ;\Gamma \left( E\right) \right) $
with respect to the product of $\mathbb{R}$-linear forms. A flat $\mathbb{R}$%
-linear connection $\nabla :\Gamma \left( A\right) \rightarrow \Gamma \left(
B\right) $ induces, denoted by $H_{\nabla ,\mathbb{R}}^{\bullet }\left(
A;B\right) $, the \emph{Lie algebroid }$\mathbb{R}$\emph{-cohomology space
with coefficients in} $B$ as the cohomology space of the complex $\left( 
\mathcal{A}lt_{\mathbb{R}}^{\bullet }\left( \Gamma \left( A\right) ;\Gamma
\left( B\right) \right) ,d_{\mathbb{R}}^{\nabla }\right) $.

The differential operator $d_{\mathbb{R}}^{\limfunc{Sec}\rho _{A}}$ induced
by the anchor, i.e. by the flat{\Huge \ }$A$-connection in $TM$, will be
denoted by $d_{A,\mathbb{R}}$. Since modules $\Gamma \left( M\times \mathbb{R%
}\right) $ and $%
%TCIMACRO{\TeXButton{C}{\mathscr{C}}}%
%BeginExpansion
\mathscr{C}%
%EndExpansion
^{\infty }\left( M\right) $ are isomorphic, it follows that $d_{A,\mathbb{R}}
$ is an extension of the exterior derivative from the space $\Omega
^{\bullet }\left( A\right) $ of ($%
%TCIMACRO{\TeXButton{C}{\mathscr{C}}}%
%BeginExpansion
\mathscr{C}%
%EndExpansion
^{\infty }\left( M\right) $-linear) differential forms on $A$ to $\mathcal{A}%
lt_{\mathbb{R}}^{\bullet }\left( \Gamma \left( A\right) ;%
%TCIMACRO{\TeXButton{C}{\mathscr{C}}}%
%BeginExpansion
\mathscr{C}%
%EndExpansion
^{\infty }\left( M\right) \right) $.

Let us recall that the cohomology space of the complex $\left( \Omega
^{\bullet }\left( A\right) ,d_{A}\right) $ where $\Omega ^{\bullet }\left(
A\right) $\text{ is}\newline
\text{the space of all }$%
%TCIMACRO{\TeXButton{C}{\mathscr{C}}}%
%BeginExpansion
\mathscr{C}%
%EndExpansion
^{\infty }\left( M\right) $-linear \text{forms on }$A$, $d_{A}=\left. d_{%
\mathbb{R}}^{\limfunc{Sec}\rho _{A}}\right\vert \Omega ^{\bullet }\left(
A\right) :\Omega ^{\bullet }\left( A\right) \rightarrow \Omega ^{\bullet
+1}\left( A\right) $, is called the \emph{cohomology of Lie algebroid} and
is denoted by $H^{\bullet }\left( A\right) $.

Let $E$ be a vector bundle over $M$. Observe that $\mathcal{A}lt_{\mathbb{R}%
}^{\bullet }\left( \Gamma \left( A\right) ;\Gamma \left( \func{End}E\right)
\right) $ is a left module over the algebra $\mathcal{A}lt_{\mathbb{R}%
}^{\bullet }\left( \Gamma \left( A\right) ;%
%TCIMACRO{\TeXButton{C}{\mathscr{C}}}%
%BeginExpansion
\mathscr{C}%
%EndExpansion
^{\infty }\left( M\right) \right) $ with the standard multiplication of
forms. Moreover, 
\begin{equation*}
\mathcal{A}lt_{\mathbb{R}}^{\bullet }\left( \Gamma \left( A\right) ;%
%TCIMACRO{\TeXButton{C}{\mathscr{C}}}%
%BeginExpansion
\mathscr{C}%
%EndExpansion
^{\infty }\left( M\right) \right) \otimes _{%
%TCIMACRO{\TeXButton{C}{\mathscr{C}}}%
%BeginExpansion
\mathscr{C}%
%EndExpansion
^{\infty }\left( M\right) }\Gamma \left( \func{End}E\right) \cong \mathcal{A}%
lt_{\mathbb{R}}^{\bullet }\left( \Gamma \left( A\right) ;\Gamma \left( \func{%
End}E\right) \right)
\end{equation*}%
as $%
%TCIMACRO{\TeXButton{C}{\mathscr{C}}}%
%BeginExpansion
\mathscr{C}%
%EndExpansion
^{\infty }\left( M\right) $-modules by the isomorphism defined in such a way
that%
\begin{equation*}
\omega \otimes \phi \longmapsto \omega \wedge \phi .
\end{equation*}

In the paper, we define the Chern-Simons forms for $\mathbb{R}$-linear
connections on Lie algebroids. The generalized Chern-Simons formula is
derived as a consequence of Stokes' formula for $\mathbb{R}$-linear forms.
The notion of the Chern classes of a vector bundle as cohomology classes of
some $\mathbb{R}$-linear Chern-Simons forms is proposed. We show that such
classes for a given $\mathbb{R}$-linear connection{\Huge \ }do not depend on
the choice of the connection. In the paper, we discuss the wider then in 
\cite{Crainic-Fernandes-jets} for linear connections set of obstructions to
the existence of a flat connection of a given Lie algebroid.

Using ideas form papers Crainic and Fernandes, we introduce the secondary
characteristic classes for arbitrary $\mathbb{R}$-linear connections of Lie
algebroids in vector bundles. If an $\mathbb{R}$-linear $A$-connection $%
\nabla $ on a vector bundle $E$ is metrizable with respect to any metric $h$
in $E$ (i.e. $\nabla h=0$), the defined secondary characteristic classes
vanishes. Therefore, secondary characteristic classes of $\nabla $ are
obstructions to the existence of an invariant metric with respect to $\nabla 
$. In \cite{Crainic-Fernandes-jets} were considered connections up to
homotopy (some non-linear connections with a local property). Here we
examine all $\mathbb{R}$-linear connections. At the end of the last section
we derive some comments on the Chern-Simons forms for $\mathbb{R}$-linear
connections (in particular for Lie algebroids over odd dimensional
manifolds).

\section{The Chern-Simons transgression forms on Lie algebroids and the
Chern Character}

Let $\left( A,\rho _{A},[\![\cdot ,\cdot ]\!]_{A}\right) $ be a Lie
algebroid on a manifold $M$, $E$ a vector bundle over $M$, $k$ a natural
number and $\func{pr}_{2}:\mathbb{R}^{k}\times M\rightarrow M$\ a projection
on the second factor. Consider an $\mathbb{R}$-linear connection $\nabla
:\Gamma \left( A\right) \rightarrow 
%TCIMACRO{\TeXButton{CDO}{\mathscr{CDO}}}%
%BeginExpansion
\mathscr{CDO}%
%EndExpansion
\left( E\right) $ of $A$ on $E$. The standard fibrewise trace $\func{Tr}%
:\Gamma \left( \func{End}E\right) \rightarrow 
%TCIMACRO{\TeXButton{C}{\mathscr{C}}}%
%BeginExpansion
\mathscr{C}%
%EndExpansion
^{\infty }\left( M\right) $ on $\func{End}\left( E\right) $ induces a trace 
\begin{equation*}
\func{Tr}_{\ast }:\mathcal{A}lt_{\mathbb{R}}^{\bullet }\left( \Gamma \left(
A\right) ;\Gamma \left( \func{End}E\right) \right) \longrightarrow \mathcal{A%
}lt_{\mathbb{R}}^{\bullet }\left( \Gamma \left( A\right) ;%
%TCIMACRO{\TeXButton{C}{\mathscr{C}}}%
%BeginExpansion
\mathscr{C}%
%EndExpansion
^{\infty }\left( M\right) \right)
\end{equation*}%
such that $\func{Tr}_{\ast }\left( \omega \right) \left( a_{1},\ldots
,a_{n}\right) =\func{Tr}\left( \left( \omega \right) \left( a_{1},\ldots
,a_{n}\right) \right) $. Set (for $p\geq 1$)%
\begin{equation*}
\func{ch}_{p}\left( \nabla \right) =\func{Tr}_{\ast }\left( R^{\nabla
}\right) ^{p}\in \mathcal{A}lt_{\mathbb{R}}^{2p}\left( \Gamma \left(
A\right) ;%
%TCIMACRO{\TeXButton{C}{\mathscr{C}}}%
%BeginExpansion
\mathscr{C}%
%EndExpansion
^{\infty }\left( M\right) \right)
\end{equation*}%
where $\left( R^{\nabla }\right) ^{p}\in \mathcal{A}lt_{\mathbb{R}%
}^{2p}\left( \Gamma \left( A\right) ;\Gamma \left( \func{End}E\right)
\right) $ is, for $a_{1},...,a_{2p}\in \Gamma \left( A\right) $, given by%
\begin{equation*}
\left( R^{\nabla }\right) ^{p}\left( a_{1},...,a_{2p}\right) =\frac{1}{2^{p}}%
\sum\nolimits_{\tau \in S_{2p}}\func{sgn}\tau \cdot R_{a_{\tau \left(
1\right) },a_{\tau \left( 2\right) }}^{\nabla }\circ \cdots \circ R_{a_{\tau
\left( 2p-1\right) },a_{\tau \left( 2p\right) }}^{\nabla }.
\end{equation*}%
The $2p$-form $\func{ch}_{p}\left( \nabla \right) $ is called the \emph{%
Chern character form }associated to $\nabla $.

\begin{lemma}
\label{comm_Tr_and_diff}$d_{A,\mathbb{R}}\circ \func{Tr}_{\ast }=\func{Tr}%
_{\ast }\circ d_{\mathbb{R}}^{\overline{\nabla }}$ where $\overline{\nabla }%
:\Gamma \left( A\right) \rightarrow 
%TCIMACRO{\TeXButton{CDO}{\mathscr{CDO}}}%
%BeginExpansion
\mathscr{CDO}%
%EndExpansion
\left( \limfunc{End}E\right) $,\emph{\ }$\overline{\nabla }_{a}=\left[
\nabla _{a},\bullet \right] $.
\end{lemma}

\begin{proof}
First, recall that the space $\mathcal{A}lt_{\mathbb{R}}^{\bullet }\left(
\Gamma \left( A\right) ;\Gamma \left( \func{End}E\right) \right) $ is
isomorphic to%
\begin{equation*}
\mathcal{A}lt_{\mathbb{R}}^{\bullet }\left( \Gamma \left( A\right) ;%
%TCIMACRO{\TeXButton{C}{\mathscr{C}}}%
%BeginExpansion
\mathscr{C}%
%EndExpansion
^{\infty }\left( M\right) \right) \otimes _{%
%TCIMACRO{\TeXButton{C}{\mathscr{C}}}%
%BeginExpansion
\mathscr{C}%
%EndExpansion
^{\infty }\left( M\right) }\Gamma \left( \func{End}E\right) .
\end{equation*}%
Let $\eta \in \mathcal{A}lt_{\mathbb{R}}^{n}\left( \Gamma \left( A\right) ;%
%TCIMACRO{\TeXButton{C}{\mathscr{C}}}%
%BeginExpansion
\mathscr{C}%
%EndExpansion
^{\infty }\left( M\right) \right) $, $\varphi \in \Gamma \left( \func{End}%
E\right) $. Then $\func{Tr}_{\ast }\left( \eta \otimes \varphi \right) =\eta
\cdot \func{Tr}\varphi $. It is a simple matter to see that $d_{A,\mathbb{R}%
}\left( \func{Tr}\varphi \right) =\func{Tr}_{\ast }\left( d_{\mathbb{R}}^{%
\overline{\nabla }}\varphi \right) $. Therefore%
\begin{eqnarray*}
d_{A,\mathbb{R}}\func{Tr}_{\ast }\left( \eta \otimes \varphi \right) &=&d_{A,%
\mathbb{R}}\eta \cdot \func{Tr}\varphi +\left( -1\right) ^{n}\eta \wedge
d_{A,\mathbb{R}}\left( \func{Tr}\varphi \right) \\
&=&d_{A,\mathbb{R}}\eta \cdot \func{Tr}\varphi +\left( -1\right) ^{n}\eta
\wedge \func{Tr}_{\ast }\left( d_{\mathbb{R}}^{\overline{\nabla }}\varphi
\right) \\
&=&\func{Tr}_{\ast }\left( d_{A,\mathbb{R}}\eta \otimes \varphi +\left(
-1\right) ^{n}\eta \wedge d_{\mathbb{R}}^{\overline{\nabla }}\varphi \right)
\\
&=&\func{Tr}_{\ast }\left( d_{\mathbb{R}}^{\overline{\nabla }}\left( \eta
\otimes \varphi \right) \right) .
\end{eqnarray*}
\end{proof}

$%
%TCIMACRO{\TeXButton{C}{\mathscr{C}}}%
%BeginExpansion
\mathscr{C}%
%EndExpansion
^{\infty }\left( \mathbb{R}\times M\right) $-modules $\Gamma \left( \func{pr}%
_{2}^{\ast }A\right) $ and $%
%TCIMACRO{\TeXButton{C}{\mathscr{C}}}%
%BeginExpansion
\mathscr{C}%
%EndExpansion
^{\infty }\left( \mathbb{R}^{k}\times M\right) \otimes _{%
%TCIMACRO{\TeXButton{C}{\mathscr{C}}}%
%BeginExpansion
\mathscr{C}%
%EndExpansion
^{\infty }\left( M\right) }\Gamma \left( A\right) $ are isomorphic (see \cite%
{Higgins-Mackenzie}) and this way the module of cross-sections of the
inverse image%
\begin{equation*}
\func{pr}_{2}^{\;\wedge }\hspace{-0.1cm}\left( A\right) =\left\{ \left(
\gamma ,w\right) \in T\left( \mathbb{R}^{k}\times M\right) \times A:\left( 
\func{pr}_{2}\right) _{\ast }\gamma =\rho _{A}\left( w\right) \right\} \cong
T\mathbb{R}^{k}\times A
\end{equation*}%
of $A$ by $\func{pr}_{2}$ is a $%
%TCIMACRO{\TeXButton{C}{\mathscr{C}}}%
%BeginExpansion
\mathscr{C}%
%EndExpansion
^{\infty }\left( \mathbb{R}^{k}\times M\right) $-submodule of 
\begin{equation*}
%TCIMACRO{\TeXButton{X}{\mathscr{X}}}%
%BeginExpansion
\mathscr{X}%
%EndExpansion
\left( \mathbb{R}^{k}\times M\right) \times \left( 
%TCIMACRO{\TeXButton{C}{\mathscr{C}}}%
%BeginExpansion
\mathscr{C}%
%EndExpansion
^{\infty }\left( \mathbb{R}^{k}\times M\right) \otimes _{%
%TCIMACRO{\TeXButton{C}{\mathscr{C}}}%
%BeginExpansion
\mathscr{C}%
%EndExpansion
^{\infty }\left( M\right) }\Gamma \left( A\right) \right) 
\end{equation*}%
($T\mathbb{R}^{k}\times A$ is the Cartesian product of Lie algebroids $T%
\mathbb{R}^{k}$ and $A$, see \cite{Kubarski-invariant}). We denote
cross-sections $0\times a$, $\frac{\partial }{\partial t^{j}}\times 0$ of
the vector bundle $T\mathbb{R}^{k}\times A$ briefly by $a$ and $\frac{%
\partial }{\partial t^{j}}$, respectively. Let 
\begin{equation*}
\Delta ^{k}=\left\{ \left( t_{1},...,t_{k}\right) \in \mathbb{R}%
^{k};\;\;\;\forall i\;\;t_{i}\geq 0\,,\;\;\sum\nolimits_{i=1}^{k}t_{i}\leq
1\right\} 
\end{equation*}%
be the \emph{standard }$k$\emph{-simplex} in $\mathbb{R}^{k}$. Additionally
we set the \emph{standard }$0$\emph{-simplex} as $\Delta ^{0}=\left\{
0\right\} $. Define%
\begin{equation*}
\dint\nolimits_{\Delta ^{k}}:\mathcal{A}lt_{\mathbb{R}}^{\bullet }\left(
\Gamma \left( T\mathbb{R}^{k}\times A\right) ;%
%TCIMACRO{\TeXButton{C}{\mathscr{C}}}%
%BeginExpansion
\mathscr{C}%
%EndExpansion
^{\infty }\left( \mathbb{R}^{k}\times M\right) \right) \longrightarrow 
\mathcal{A}lt_{\mathbb{R}}^{\bullet -k}\left( \Gamma \left( A\right) ;%
%TCIMACRO{\TeXButton{C}{\mathscr{C}}}%
%BeginExpansion
\mathscr{C}%
%EndExpansion
^{\infty }\left( M\right) \right) ,
\end{equation*}%
\begin{equation*}
\left( \int\nolimits_{\Delta ^{k}}\omega \right) \left(
a_{1},...,a_{n-k}\right) =\int\nolimits_{\Delta ^{k}}\omega \left( \frac{%
\partial }{\partial t^{1}},...,\frac{\partial }{\partial t^{k}}%
,a_{1},...,a_{n-k}\right) _{|\left( t_{1},...,t_{k},\bullet \right)
}dt_{1}...dt_{k},
\end{equation*}%
\begin{equation*}
\left( \int\nolimits_{\Delta ^{0}}\omega \right) \left(
a_{1},...,a_{n}\right) =\iota _{0}^{\ast }\left( \omega \left( 0\times
a_{1},...,0\times a_{n}\right) \right) ,\ \ \ \int\nolimits_{\Delta
^{0}}f=\iota _{0}^{\ast }f
\end{equation*}%
for all $n\geq 1$, $1\leq k\leq n$, $\omega \in \mathcal{A}lt_{\mathbb{R}%
}^{n}\left( \Gamma \left( T\mathbb{R}^{k}\times A\right) ;%
%TCIMACRO{\TeXButton{C}{\mathscr{C}}}%
%BeginExpansion
\mathscr{C}%
%EndExpansion
^{\infty }\left( M\right) \right) $, $f\in 
%TCIMACRO{\TeXButton{C}{\mathscr{C}}}%
%BeginExpansion
\mathscr{C}%
%EndExpansion
^{\infty }\left( \mathbb{R}^{k}\times M\right) $ and where $\iota
_{0}:M\rightarrow \Delta ^{0}\times M$ is an inclusion defined by $\iota
_{0}\left( x\right) =\left( 0,x\right) $.

In view of the factorization property in $%
%TCIMACRO{\TeXButton{C}{\mathscr{C}}}%
%BeginExpansion
\mathscr{C}%
%EndExpansion
^{\infty }\left( \mathbb{R}^{k}\times M\right) \otimes _{%
%TCIMACRO{\TeXButton{C}{\mathscr{C}}}%
%BeginExpansion
\mathscr{C}%
%EndExpansion
^{\infty }\left( M\right) }\Gamma \left( A\right) $, we conclude that for $%
\nabla $ there exists exactly one $\mathbb{R}$-linear connection 
\begin{equation*}
\widetilde{\nabla }:\Gamma \left( T\mathbb{R}^{k}\times A\right)
\longrightarrow 
%TCIMACRO{\TeXButton{CDO}{\mathscr{CDO}}}%
%BeginExpansion
\mathscr{CDO}%
%EndExpansion
\left( \func{pr}_{2}^{\;\ast }E\right) 
\end{equation*}%
of $T\mathbb{R}^{k}\times A$ on $\func{pr}_{2}^{\;\ast }E$\emph{\ }such that%
\begin{equation*}
\left( \widetilde{\nabla }_{\left( X,\tsum_{i}r^{i}\otimes a^{i}\right)
}\left( \nu \circ \func{pr}_{2}\right) \right) \left( t,\bullet \right)
=\nabla _{\tsum_{i}r^{i}\left( t,\bullet \right) \cdot a^{i}}\left( \nu
\right) 
\end{equation*}%
for all\ $\left( X,\tsum_{i}r^{i}\otimes a^{i}\right) \in 
%TCIMACRO{\TeXButton{X}{\mathscr{X}}}%
%BeginExpansion
\mathscr{X}%
%EndExpansion
\left( \mathbb{R}^{k}\times M\right) \times \left( 
%TCIMACRO{\TeXButton{C}{\mathscr{C}}}%
%BeginExpansion
\mathscr{C}%
%EndExpansion
^{\infty }\left( \mathbb{R}^{k}\times M\right) \otimes _{%
%TCIMACRO{\TeXButton{C}{\mathscr{C}}}%
%BeginExpansion
\mathscr{C}%
%EndExpansion
^{\infty }\left( M\right) }\Gamma \left( A\right) \right) $,\ $\nu \in
\Gamma \left( E\right) $, $t=\left( t_{1},...,t_{k}\right) \in \mathbb{R}^{k}
$.\ In particular, $\left( \widetilde{\nabla }_{\left( 0\times \left( \rho
_{A}\circ a\right) ,1\otimes a\right) }\left( \nu \circ \func{pr}_{2}\right)
\right) \left( t,\bullet \right) =\nabla _{a}\left( \nu \right) $, $a\in
\Gamma \left( A\right) $.$\;$The connection$\;\widetilde{\nabla }$\ is
called the \emph{lifting}\ of $\nabla $ to $T\mathbb{R}^{k}\times A$.

Let\ $\nabla ^{0},\,\nabla ^{1},\ldots ,\nabla ^{k}:\Gamma \left( A\right)
\rightarrow 
%TCIMACRO{\TeXButton{CDO}{\mathscr{CDO}}}%
%BeginExpansion
\mathscr{CDO}%
%EndExpansion
\left( E\right) \;$be $\mathbb{R}$-linear connections of a Lie algebroid $A$
on a vector bundle $E$\ and\ $\widetilde{\nabla }^{0},\,\widetilde{\nabla }%
^{1},\ldots ,\,\widetilde{\nabla }^{k}:\Gamma \left( T\mathbb{R}^{k}\times
A\right) \rightarrow 
%TCIMACRO{\TeXButton{CDO}{\mathscr{CDO}}}%
%BeginExpansion
\mathscr{CDO}%
%EndExpansion
\left( \func{pr}_{2}^{\;\ast }E\right) $\ be their liftings to $T\mathbb{R}%
^{k}\times A$.\ Then there exists an $\mathbb{R}$-linear connection%
\begin{equation*}
\nabla ^{\func{aff}_{k}}:\Gamma \left( T\mathbb{R}^{k}\times A\right)
\longrightarrow 
%TCIMACRO{\TeXButton{CDO}{\mathscr{CDO}}}%
%BeginExpansion
\mathscr{CDO}%
%EndExpansion
\left( \func{pr}_{2}^{\;\ast }E\right) ,
\end{equation*}%
called the \emph{affine combination of connections} $\nabla ^{0},\nabla
^{1},\ldots ,\nabla ^{k}$, given by 
\begin{eqnarray*}
&&\left( \nabla _{\,\,\,\left( X,\tsum_{i}r^{i}\otimes a^{i}\right) }^{\func{%
aff}_{k}}\left( \nu \circ \func{pr}_{2}\right) \right) \left( t,\bullet
\right)  \\
&=&\left( 1-\sum\nolimits_{i=1}^{k}t_{i}\right) \cdot \left( \nabla
^{0}\right) _{\tsum_{i}r^{i}\left( t,\bullet \right) \cdot a^{i}}\left( \nu
\right) +\sum\nolimits_{i=1}^{k}t_{i}\cdot \left( \nabla ^{i}\right)
_{\tsum_{i}r^{i}\left( t,\bullet \right) \cdot a^{i}}\left( \nu \right) .
\end{eqnarray*}%
For all $0<k\leq 2p$ we define an $\mathbb{R}$-linear form%
\begin{equation*}
\limfunc{cs}\nolimits_{p}\left( \nabla ^{0},...,\nabla ^{k}\right)
=\int\nolimits_{\Delta ^{k}}\func{ch}_{p}\left( \nabla ^{\func{aff}%
_{k}}\right) \in \mathcal{A}lt_{\mathbb{R}}^{2p-k}\left( \Gamma \left(
A\right) ;%
%TCIMACRO{\TeXButton{C}{\mathscr{C}}}%
%BeginExpansion
\mathscr{C}%
%EndExpansion
^{\infty }\left( M\right) \right) 
\end{equation*}%
called the \emph{Chern-Simons form} for $\left( \nabla ^{0},...,\nabla
^{k}\right) $ and additionally we put $cs_{p}\left( \nabla ^{0}\right) =%
\func{ch}_{p}\left( \nabla ^{0}\right) $.

We have the following (useful) Stokes' formula for $\mathbb{R}$-linear forms
on $A$ (see \cite{Balcerzak-Stokes}) being a generalization of the one for
tangent bundles given by R.~Bott \cite{Bott}. For every natural number\emph{%
\ }$k$,%
\begin{equation}
\int\nolimits_{\Delta ^{k}}\circ \,d_{T\mathbb{R}^{k}\times A,\mathbb{R}%
}+\left( -1\right) ^{k+1}d_{A,\mathbb{R}}\circ \int\nolimits_{\Delta
^{k}}=\dsum\nolimits_{j=0}^{k}\left( -1\right) ^{j}\int\nolimits_{\Delta
^{k-1}}\circ \,\left( d\sigma _{j}^{k-1}\times \func{id}_{A}\right) ^{\ast },
\label{Stokes}
\end{equation}%
where $\sigma _{j}^{k}:\mathbb{R}^{k}\rightarrow \mathbb{R}^{k+1}$\ for $%
0\leq j\leq k+1$\ are functions defined by $\sigma _{0}^{0}\left( 0\right)
=1 $, $\sigma _{1}^{0}\left( 0\right) =0$, and for $t=\left(
t_{1},...,t_{k}\right) \in \mathbb{R}^{k}$ by%
\begin{eqnarray*}
\sigma _{0}^{k}\left( t\right) &=&\left(
1-\dsum\nolimits_{i=1}^{k}t_{i},t_{1},...,t_{k}\right) , \\
\sigma _{j}^{k}\left( t\right) &=&\left(
t_{1},...,t_{j-1},0,t_{j},...,t_{k}\right) ,\;\;1\leq j\leq k+1,
\end{eqnarray*}%
and where $\left( \left( \int\nolimits_{\Delta ^{k-1}}\circ \,\left( d\sigma
_{j}^{k-1}\times \func{id}_{A}\right) ^{\ast }\right) \omega \right) \left(
a_{1},...,a_{n-k+1}\right) $ is, by definition, equal to%
\begin{equation*}
\int\nolimits_{\Delta ^{k-1}}\omega \left( d\sigma _{j}^{k-1}\left( \frac{%
\partial }{\partial t^{1}}\right) ,...,d\sigma _{j}^{k-1}\left( \frac{%
\partial }{\partial t^{k-1}}\right) ,a_{1},...,a_{n-k+1}\right) _{|\left(
t_{1},...,t_{k-1},\bullet \right) }\hspace{-0.3cm}dt_{1}...dt_{k-1}
\end{equation*}%
and 
\begin{equation*}
\left( \left( \int\nolimits_{\Delta ^{0}}\circ \,\left( d\sigma
_{j}^{0}\times \func{id}_{A}\right) ^{\ast }\right) \omega \right) \left(
a_{1},...,a_{n}\right) =\left( \sigma _{j}^{0}\times \func{id}_{M}\circ
\iota _{0}\right) ^{\ast }\left( \omega \left( a_{1},...,a_{n}\right) \right)
\end{equation*}%
if $k\geq 2$, $\omega \in \mathcal{A}lt_{\mathbb{R}}^{n}\left( \Gamma \left(
T\mathbb{R}^{k}\times A\right) ;%
%TCIMACRO{\TeXButton{C}{\mathscr{C}}}%
%BeginExpansion
\mathscr{C}%
%EndExpansion
^{\infty }\left( \mathbb{R}^{k}\times M\right) \right) $, $a_{i}\in \Gamma
\left( A\right) $, $j\in \left\{ 0,1\right\} $.

The following lemma will be useful below in the proof of the Chern-Simons
formula for $\mathbb{R}$-linear connections of Lie algebroids.

\begin{lemma}
\label{lemma_abc}Let $a,\,b\in \Gamma \left( A\right) $,\ $\nu \in \Gamma
\left( E\right) $,\ $t\in \mathbb{R}^{k-1}$, $0\leq j\leq k,$ $1\leq s\leq
k, $ $1\leq z\leq k-1$. Denote here the affine combination $\nabla ^{\func{%
aff}_{k}}$ of$\ \nabla ^{0},\ldots ,\nabla ^{k}$ by $\nabla ^{0,...,k}$. Then

\begin{itemize}
\item[\emph{(a)}] $\left( R_{a,b}^{\nabla ^{0,...,k}}\left( \nu \circ \func{%
pr}_{2}\right) \right) \left( \sigma _{j}^{k-1}\left( t\right) ,\bullet
\right) =\left( R_{a,b}^{\nabla ^{0,...\widehat{j}...,k}}\left( \nu \circ 
\func{pr}_{2}\right) \right) \left( t,\bullet \right) ,$

\item[\emph{(b)}] $\left( R_{\frac{\partial }{\partial \,\tilde{t}\,^{s}}%
,a}^{\nabla ^{0,...,k}}\left( \nu \circ \func{pr}_{2}\right) \right) \left(
\sigma _{j}^{k-1}\left( t\right) ,\bullet \right) $ is equal to $\left( R_{%
\frac{\partial }{\partial \,t^{s}},a}^{\nabla ^{0,...\widehat{j}%
...,k}}\left( \nu \circ \func{pr}_{2}\right) \right) \left( t,\bullet
\right) $ if $1\leq s<j,$ and $\left( R_{\frac{\partial }{\partial \,t^{s-1}}%
,a}^{\nabla ^{0,...\widehat{j}...,k}}\left( \nu \circ \func{pr}_{2}\right)
\right) \left( t,\bullet \right) $ if $j\leq s\leq k$, and where $\tilde{t}%
\,^{i}$ are coordinates of the identity map of $\mathbb{R}^{k}$,

\item[\emph{(c)}] $\left( R_{d\sigma _{j}^{k-1}\left( \frac{\partial }{%
\partial t^{z}}\right) ,a}^{\nabla ^{0,...,k}}\left( \nu \circ \func{pr}%
_{2}\right) \right) \left( \sigma _{j}^{k-1}\left( t\right) ,\bullet \right)
=\left( R_{\frac{\partial }{\partial t^{z}},a}^{\nabla ^{1...,k}}\left( \nu
\circ \func{pr}_{2}\right) \right) \left( t,\bullet \right) $.
\end{itemize}
\end{lemma}

\begin{proof}
Just calculations.
\end{proof}

\begin{theorem}
\emph{(The Chern-Simons formula for Lie algebroids and }$\mathbb{R}$\emph{%
-linear connections)} Let $\left( A,\rho _{A},[\![\cdot ,\cdot ]\!]\right) $
be a Lie algebroid on a manifold $M$, $E$ a vector bundle over $M$, $k\in 
\mathbb{N}$, $\nabla ^{0},\,...,\,\nabla ^{k}:\Gamma \left( A\right)
\rightarrow 
%TCIMACRO{\TeXButton{CDO}{\mathscr{CDO}}}%
%BeginExpansion
\mathscr{CDO}%
%EndExpansion
\left( E\right) \;$ $\mathbb{R}$-linear connections of $A$ on $E$. Then%
\begin{equation}
\left( -1\right) ^{k+1}\,d_{A,\mathbb{R}}\left( \limfunc{cs}%
\nolimits_{p}\left( \nabla ^{0},...,\nabla ^{k}\right) \right)
=\sum\nolimits_{j=0}^{k}\left( -1\right) ^{j}\,\limfunc{cs}%
\nolimits_{p}\left( \nabla ^{0},...\widehat{\nabla ^{j}}...,\nabla
^{k}\right)  \label{ChernSimonsformula}
\end{equation}%
for all integer numbers $p$ such that $0<k\leq 2p$ and $d_{A,\mathbb{R}%
}\left( cs_{p}\left( \nabla ^{0}\right) \right) =0$.
\end{theorem}

\begin{proof}
From Lemma \ref{comm_Tr_and_diff} and the Bianchi identity ($d_{\mathbb{R}}^{%
\overline{\nabla ^{j}}}\left( R^{\nabla ^{j}}\right) =0$) we deduce that
forms $\func{ch}_{p}\left( \nabla ^{0}\right) $ and $\func{ch}_{p}\left(
\nabla ^{\func{aff}_{k}}\right) $ are closed. Since these forms are closed,
applying the Stokes formula (\ref{Stokes}) we conclude that 
\begin{equation*}
\left( -1\right) ^{k+1}\,d_{A,\mathbb{R}}\left( \limfunc{cs}%
\nolimits_{p}\left( \nabla ^{0},...,\nabla ^{k}\right) \right)
=\sum\limits_{j=0}^{k}\left( -1\right) ^{j}\int\nolimits_{\Delta
^{k-1}}\left( d\sigma _{j}^{k-1}\times \func{id}_{A}\right) ^{\ast }\func{ch}%
_{p}\left( \nabla ^{\func{aff}_{k}}\right) .
\end{equation*}%
Let $a_{0},$...,$a_{2p-k}\in \Gamma \left( A\right) $. From the above%
\begin{multline*}
\left( -1\right) ^{k+1}\,d_{A,\mathbb{R}}\left( \limfunc{cs}%
\nolimits_{p}\left( \nabla ^{0},...,\nabla ^{k}\right) \right) \left(
a_{0},...,a_{2p-k}\right) \\
=\dsum\nolimits_{j=0}^{k}\left( -1\right) ^{j}\left( \int\nolimits_{\Delta
^{k-1}}\left( d\sigma _{j}^{k-1}\times \func{id}_{M}\right) ^{\ast }\func{ch}%
_{p}\left( \nabla ^{\func{aff}_{k}}\right) \right) \left(
a_{0},...,a_{2p-k}\right) .
\end{multline*}%
From the definition of $\left( R^{\nabla ^{\func{aff}_{k}}}\right) ^{p}$ and
fact that $R_{\frac{\partial }{\partial \widetilde{t}^{i}},\frac{\partial }{%
\partial \widetilde{t}^{j}}}^{\nabla ^{\func{aff}_{k}}}=0$ (where $\left( 
\tilde{t}\,^{1},...,\tilde{t}\,^{k}\right) $ is the identity map on the
manifold $\mathbb{R}^{k}$) we observe that the possible non-zero terms in
the above sum are the form%
\begin{equation*}
R_{d\sigma _{j}^{k-1}\left( \frac{\partial }{\partial t^{s}}\right)
,a}^{\nabla ^{\func{aff}_{k}}}\circ \cdots \circ R_{b,c}^{\nabla ^{\func{aff}%
_{k}}}\circ \cdots \circ R_{d,e}^{\nabla ^{\func{aff}_{k}}},\ \ \text{\ }%
a,b,c,d,e\in \Gamma \left( A\right) .
\end{equation*}%
Lemma \ref{lemma_abc} now yields that $\left( -1\right) ^{k+1}\,d_{A,\mathbb{%
R}}\left( \limfunc{cs}\nolimits_{p}\left( \nabla ^{0},...,\nabla ^{k}\right)
\right) \left( a_{0},...,a_{2p-k}\right) $ is equal to%
\begin{gather*}
\sum_{j=0}^{k}\left( -1\right) ^{j}\hspace{-0.1cm}\int\nolimits_{\Delta
^{k-1}}\hspace{-0.3cm}\func{ch}_{p}\left( \nabla ^{0,...\widehat{j}%
...,k}\right) \hspace{-0.1cm}\left. \left( \frac{\partial }{\partial t^{1}}%
,...,\frac{\partial }{\partial t^{k-1}},a_{0\,},...,a_{2p-k}\right)
\right\vert _{\left( t_{1},...,t_{k-1},\bullet \right) }\hspace{-0.3cm}%
dt_{1}\ldots dt_{k-1} \\
=\left( \sum_{j=0}^{k}\left( -1\right) ^{j}\,\limfunc{cs}\nolimits_{p}\left(
\nabla _{0},...\widehat{\nabla ^{j}}...,\nabla _{k}\right) \right) \left(
a_{0\,},...,a_{2p-k}\right) .
\end{gather*}
\end{proof}

\begin{remark}
\emph{If }$\nabla ^{0},\nabla ^{1},\ldots ,\nabla ^{k}:\Gamma \left(
A\right) \rightarrow 
%TCIMACRO{\TeXButton{CDO}{\mathscr{CDO}}}%
%BeginExpansion
\mathscr{CDO}%
%EndExpansion
\left( E\right) \;$\emph{are }$%
%TCIMACRO{\TeXButton{C}{\mathscr{C}}}%
%BeginExpansion
\mathscr{C}%
%EndExpansion
^{\infty }\left( M\right) $\emph{-linear connections, then }$\nabla ^{\func{%
aff}_{k}}$\emph{\ is a }$%
%TCIMACRO{\TeXButton{C}{\mathscr{C}}}%
%BeginExpansion
\mathscr{C}%
%EndExpansion
^{\infty }\left( M\right) $\emph{-linear connection. In this case, we obtain
a formula due to property of Chern-Simons transgressions in \cite%
{Crainic-Fernandes-jets} by M.~Crainic and R.~L.~Fernandes.}
\end{remark}

\begin{corollary}
\emph{Let }$\nabla :\Gamma \left( A\right) \rightarrow 
%TCIMACRO{\TeXButton{CDO}{\mathscr{CDO}}}%
%BeginExpansion
\mathscr{CDO}%
%EndExpansion
\left( E\right) $ \emph{be an} $\mathbb{R}$\emph{-linear connection of }$A$%
\emph{\ on a vector bundle }$E$\emph{.} \emph{The Chern character forms} $%
\func{ch}_{p}\left( \nabla \right) \in \mathcal{A}lt_{\mathbb{R}}^{2p}\left(
\Gamma \left( A\right) ;%
%TCIMACRO{\TeXButton{C}{\mathscr{C}}}%
%BeginExpansion
\mathscr{C}%
%EndExpansion
^{\infty }\left( M\right) \right) $ \emph{are closed and their cohomology
classes}%
\begin{equation*}
\func{ch}_{p}\left( A,E\right) =\left[ \func{ch}_{p}\left( \nabla \right) %
\right] \in H_{\rho _{A},\mathbb{R}}^{2p}\left( A;M\times \mathbb{R}\right) 
\emph{,}
\end{equation*}%
\emph{do not depend on the choice of the connection }$\nabla $\emph{.
Indeed, let} $\nabla ^{0}$,\thinspace $\nabla ^{1}:\Gamma \left( A\right)
\rightarrow 
%TCIMACRO{\TeXButton{CDO}{\mathscr{CDO}}}%
%BeginExpansion
\mathscr{CDO}%
%EndExpansion
\left( E\right) $ \emph{be} $\mathbb{R}$\emph{-linear connections of }$A$%
\emph{\ on }$E$\emph{. According to (\ref{ChernSimonsformula}), we have}%
\begin{eqnarray*}
d_{A,\mathbb{R}}\left( \limfunc{cs}\nolimits_{p}\left( \nabla ^{0},\nabla
^{1}\right) \right)  &=&\limfunc{cs}\nolimits_{p}\left( \nabla ^{1}\right) -%
\limfunc{cs}\nolimits_{p}\left( \nabla ^{0}\right)  \\
&=&\func{ch}_{p}\left( \nabla ^{1}\right) -\func{ch}_{p}\left( \nabla
^{0}\right) .
\end{eqnarray*}
\end{corollary}

In this way we have correctly defined the Chern character%
\begin{equation*}
\func{ch}\left( A,E\right) \in H_{\rho _{A},\mathbb{R}}\left( A;M\times 
\mathbb{R}\right) .
\end{equation*}

\begin{remark}
\emph{(\cite{Crainic-up to homotopy}, \cite{Crainic-Fernandes-jets}) In the
particular case we can obtain the Chern character for a non-linear
connection }$\nabla :\Gamma \left( A\right) \rightarrow 
%TCIMACRO{\TeXButton{CDO}{\mathscr{CDO}}}%
%BeginExpansion
\mathscr{CDO}%
%EndExpansion
\left( E\right) $\emph{\ of a Lie algebroid }$A$\emph{\ on a vector bundle }$%
E$\emph{, i.e. a local} $\mathbb{R}$\emph{-linear connection }$\nabla
:\Gamma \left( A\right) \rightarrow 
%TCIMACRO{\TeXButton{CDO}{\mathscr{CDO}}}%
%BeginExpansion
\mathscr{CDO}%
%EndExpansion
\left( E\right) $\emph{. In the space }$\Omega _{nl}\left( A\right) $\emph{\
of non-linear differential forms on }$A$\emph{\ (local }$\mathbb{R}$\emph{%
-linear forms on }$A$\emph{) we have the differential operator }$d_{nl}=d_{A,%
\mathbb{R}}|\Omega _{nl}^{\bullet }\left( A\right) :$\emph{\ }$\Omega
_{nl}^{\bullet }\left( A\right) \rightarrow \Omega _{nl}^{\bullet +1}\left(
A\right) $.
\end{remark}

\section{Secondary characteristic classes for $\mathbb{R}$-linear
connections and some the Chern-Simons forms for a pair of connections}

Let $E$ be a vector bundle over $M$ with a metric $h$ and $\nabla :\Gamma
\left( A\right) \rightarrow 
%TCIMACRO{\TeXButton{CDO}{\mathscr{CDO}}}%
%BeginExpansion
\mathscr{CDO}%
%EndExpansion
\left( E\right) $ be an $\mathbb{R}$-linear connection of a Lie algebroid $A$
on $E$. We define an $\mathbb{R}$-linear connection $\nabla ^{h}:\Gamma
\left( A\right) \rightarrow 
%TCIMACRO{\TeXButton{CDO}{\mathscr{CDO}}}%
%BeginExpansion
\mathscr{CDO}%
%EndExpansion
\left( E\right) $ of $A$ on $E$ such that%
\begin{equation*}
\left( \rho _{A}\circ a\right) \left( h\left( s,t\right) \right) =h\left(
\nabla _{a}s,t\right) +h\left( s,\nabla _{a}^{h}t\right) ,\ \ \ a\in \Gamma
\left( A\right) ,\ s,t\in \Gamma \left( E\right) .
\end{equation*}%
We can observe that%
\begin{equation*}
R_{a,b}^{\nabla ^{h}}=-\left( R_{a,b}^{\nabla }\right) ^{\ast },\ \ \ \ \ \
a,b\in \Gamma \left( A\right) ,
\end{equation*}%
where $\left( R_{a,b}^{\nabla }\right) ^{\ast }$ is the adjoint map to $%
R_{a,b}^{\nabla }$ with respect to $h$. Therefore we obtain the following
lemma.

\begin{lemma}
\label{Lemma_1_about_nabla_h}If $\nabla _{0}$, $\nabla _{1}$ are $\mathbb{R}$%
-linear connections of $A$ on $E$, then

\begin{itemize}
\item[(a)] $\limfunc{cs}\nolimits_{p}\left( \nabla _{0}^{h}\right) =\left(
-1\right) ^{p}\limfunc{cs}\nolimits_{p}\left( \nabla _{0}\right) $,

\item[(b)] $\limfunc{cs}\nolimits_{p}\left( \nabla _{0}^{h},\nabla
_{1}^{h}\right) =\left( -1\right) ^{p}\limfunc{cs}\nolimits_{p}\left( \nabla
_{0},\nabla _{1}\right) .$
\end{itemize}
\end{lemma}

From the Chern-Simons formula (\ref{ChernSimonsformula}) and Lemma \ref%
{Lemma_1_about_nabla_h} (a) we deduce that 
\begin{eqnarray*}
d_{A,\mathbb{R}}\limfunc{cs}\nolimits_{p}\left( \nabla ,\nabla ^{h}\right)
&=&\limfunc{cs}\nolimits_{p}\left( \nabla \right) -\limfunc{cs}%
\nolimits_{p}\left( \nabla ^{h}\right) \\
&=&\limfunc{cs}\nolimits_{p}\left( \nabla \right) -\left( -1\right) ^{p}%
\limfunc{cs}\nolimits_{p}\left( \nabla \right) \\
&=&0,
\end{eqnarray*}%
because $\nabla $ is flat. In particular, we see that $\nabla ^{h}$ is also
flat.

\begin{theorem}
The cohomology class $\left[ \limfunc{cs}\nolimits_{p}\left( \nabla ,\nabla
^{h}\right) \right] \in H_{\rho _{A},\mathbb{R}}^{2p-1}\left( A\right) $ do
not depend on the choice of metric $h$.
\end{theorem}

\begin{proof}
Let $h_{1}$, $h_{2}$ be two metrics on $E$ and let $\nabla ^{M}$ be any $TM$%
-connection on $E$. Thus $\nabla _{o}=\nabla ^{M}\circ \rho _{A}$ is an $A$%
-connection on $E$ (i.e. a linear connection). The Chern-Simons formula (\ref%
{ChernSimonsformula}) yields%
\begin{equation}
-d_{A,\mathbb{R}}\limfunc{cs}\nolimits_{p}\left( \nabla ,\nabla
^{h_{j}},\nabla _{o}^{h_{j}}\right) =\limfunc{cs}\nolimits_{p}\left( \nabla
^{h_{j}},\nabla _{o}^{h_{j}}\right) -\limfunc{cs}\nolimits_{p}\left( \nabla
,\nabla _{o}^{h_{j}}\right) +\limfunc{cs}\nolimits_{p}\left( \nabla ,\nabla
^{h_{j}}\right)  \label{e1}
\end{equation}%
and%
\begin{equation}
-d_{A,\mathbb{R}}\limfunc{cs}\nolimits_{p}\left( \nabla ,\nabla _{o},\nabla
_{o}^{h_{j}}\right) =\limfunc{cs}\nolimits_{p}\left( \nabla _{o},\nabla
_{o}^{h_{j}}\right) -\limfunc{cs}\nolimits_{p}\left( \nabla ,\nabla
_{o}^{h_{j}}\right) +\limfunc{cs}\nolimits_{p}\left( \nabla ,\nabla
_{o}\right)  \label{e2}
\end{equation}%
for $j\in \left\{ 1,2\right\} $. Lemma \ref{Lemma_1_about_nabla_h} implies $%
\limfunc{cs}\nolimits_{p}\left( \nabla ^{h_{j}},\nabla _{o}^{h_{j}}\right)
=\left( -1\right) ^{p}\limfunc{cs}\nolimits_{p}\left( \nabla ,\nabla
_{o}\right) $. From this, (\ref{e1}) and (\ref{e2}) we get%
\begin{align*}
& \limfunc{cs}\nolimits_{p}\left( \nabla ,\nabla ^{h_{1}}\right) -\limfunc{cs%
}\nolimits_{p}\left( \nabla ,\nabla ^{h_{2}}\right) \\
& =d_{A,\mathbb{R}}\left( \limfunc{cs}\nolimits_{p}\left( \nabla ,\nabla
^{h_{2}},\nabla _{o}^{h_{2}}\right) -\limfunc{cs}\nolimits_{p}\left( \nabla
,\nabla ^{h_{1}},\nabla _{o}^{h_{1}}\right) \right) +\limfunc{cs}%
\nolimits_{p}\left( \nabla ,\nabla _{o}^{h_{1}}\right) -\limfunc{cs}%
\nolimits_{p}\left( \nabla ,\nabla _{o}^{h_{2}}\right) \\
& =d_{A,\mathbb{R}}\left( \limfunc{cs}\nolimits_{p}\left( \nabla ,\nabla
^{h_{2}},\nabla _{o}^{h_{2}}\right) -\limfunc{cs}\nolimits_{p}\left( \nabla
,\nabla ^{h_{1}},\nabla _{o}^{h_{1}}\right) \right) +d_{A,\mathbb{R}}%
\limfunc{cs}\nolimits_{p}\left( \nabla ,\nabla _{o},\nabla
_{o}^{h_{1}}\right) \\
& -d_{A,\mathbb{R}}\limfunc{cs}\nolimits_{p}\left( \nabla ,\nabla
_{o},\nabla _{o}^{h_{2}}\right) +\limfunc{cs}\nolimits_{p}\left( \nabla
_{o},\nabla _{o}^{h_{1}}\right) -\limfunc{cs}\nolimits_{p}\left( \nabla
_{o},\nabla _{o}^{h_{2}}\right) .
\end{align*}%
Because of $\nabla _{o}$ is a linear connection, Proposition 1 from \cite%
{Crainic-Fernandes-jets} yields $\limfunc{cs}\nolimits_{p}\left( \nabla
_{o},\nabla _{o}^{h_{1}}\right) $\newline
$-\limfunc{cs}\nolimits_{p}\left( \nabla _{o},\nabla _{o}^{h_{2}}\right) $
is an exact form. In this way cohomology classes of $\limfunc{cs}%
\nolimits_{p}\left( \nabla ,\nabla ^{h_{1}}\right) $ and $\limfunc{cs}%
\nolimits_{p}\left( \nabla ,\nabla ^{h_{2}}\right) $ are both equal.
\end{proof}

\begin{definition}
\emph{We call }%
\begin{equation*}
\limfunc{u}\nolimits_{2p-1}\left( A,E\right) =\left[ \limfunc{cs}%
\nolimits_{p}\left( \nabla ,\nabla ^{h}\right) \right] \in H_{\rho _{A},%
\mathbb{R}}^{2p-1}\left( A\right) ,\ \ \ p\in \left\{ 1,\ldots ,\limfunc{rank%
}E\right\} ,
\end{equation*}%
the secondary characteristic classes \emph{of an }$\mathbb{R}$\emph{-linear
connection }$\nabla :\Gamma \left( A\right) \rightarrow 
%TCIMACRO{\TeXButton{CDO}{\mathscr{CDO}}}%
%BeginExpansion
\mathscr{CDO}%
%EndExpansion
\left( E\right) $.
\end{definition}

If there exists in $E$ an invariant metric $h$ with respect to $\nabla $,
then $\nabla ^{h}=\nabla $. Then classes $\limfunc{u}\nolimits_{2p-1}\left(
A,E\right) $ are equal to zero. Hence these classes are obstructions to the
existence of an invariant metric with respect to $\nabla $.

We obtain the following theorem analogous to Proposition 2 in \cite%
{Crainic-Fernandes-jets}.

\begin{theorem}
Let $\nabla $, $\nabla _{m}$ be $\mathbb{R}$-linear connections of $A$ on $E$
and $\nabla _{m}$ be additionally metric.

\begin{itemize}
\item[(a)] If $p$ is even, then $\limfunc{u}\nolimits_{2p-1}\left(
A,E\right) =0$.

\item[(b)] If $p$ is odd, then $\limfunc{cs}\nolimits_{p}\left( \nabla
,\nabla _{m}\right) $ is a closed form and%
\begin{equation*}
\limfunc{u}\nolimits_{2p-1}\left( A,E\right) =\left[ 2\limfunc{cs}%
\nolimits_{p}\left( \nabla ,\nabla _{m}\right) \right] .
\end{equation*}
\end{itemize}
\end{theorem}

\begin{proof}
Let $\nabla _{m}$ be metric connection with respect to a metric $h$. On
account of the Chern-Simons formula (\ref{ChernSimonsformula}), we have%
\begin{equation*}
-d_{A,\mathbb{R}}\limfunc{cs}\nolimits_{p}\left( \nabla ,\nabla ^{h},\nabla
_{m}\right) =\limfunc{cs}\nolimits_{p}\left( \nabla ^{h},\nabla _{m}\right) -%
\limfunc{cs}\nolimits_{p}\left( \nabla ,\nabla _{m}\right) +\limfunc{cs}%
\nolimits_{p}\left( \nabla ,\nabla ^{h}\right) .
\end{equation*}%
Now Lemma \ref{Lemma_1_about_nabla_h} leads to $\limfunc{cs}%
\nolimits_{p}\left( \nabla ^{h},\nabla _{m}\right) =\left( -1\right) ^{p}%
\limfunc{cs}\nolimits_{p}\left( \nabla ,\nabla _{m}\right) $, because $%
\nabla _{m}^{h}=\nabla _{m}$. It follows that%
\begin{eqnarray*}
\limfunc{cs}\nolimits_{p}\left( \nabla ,\nabla ^{h}\right) &=&\limfunc{cs}%
\nolimits_{p}\left( \nabla ,\nabla _{m}\right) -\limfunc{cs}%
\nolimits_{p}\left( \nabla ^{h},\nabla _{m}\right) -d_{A,\mathbb{R}}\limfunc{%
cs}\nolimits_{p}\left( \nabla ,\nabla ^{h},\nabla _{m}\right) \\
&=&\left( 1+\left( -1\right) ^{p+1}\right) \limfunc{cs}\nolimits_{p}\left(
\nabla ,\nabla _{m}\right) -d_{A,\mathbb{R}}\limfunc{cs}\nolimits_{p}\left(
\nabla ,\nabla ^{h},\nabla _{m}\right) ,
\end{eqnarray*}%
which completes the proof.
\end{proof}

\bigskip

For two $\mathbb{R}$-linear connections $\nabla ^{0}$,\thinspace $\nabla
^{1}:\Gamma \left( A\right) \rightarrow 
%TCIMACRO{\TeXButton{CDO}{\mathscr{CDO}}}%
%BeginExpansion
\mathscr{CDO}%
%EndExpansion
\left( E\right) $ of $A$ on $E$ we define an $\mathbb{R}$-linear $1$-form%
\begin{equation*}
\lambda =\nabla ^{1}-\nabla ^{0}\in \mathcal{A}lt_{\mathbb{R}}^{1}\left(
\Gamma \left( A\right) ;\Gamma \left( \func{End}E\right) \right) .
\end{equation*}%
Let us observe that 
\begin{equation}
R^{\nabla ^{1}}=R^{\nabla ^{0}}+d^{\overline{\nabla }^{0}}\lambda +\left[
\lambda ,\lambda \right] ,  \label{curvature_and_pair}
\end{equation}%
where $d^{\overline{\nabla }^{0}}$is the covariant derivative in $\mathcal{A}%
lt_{\mathbb{R}}^{\bullet }\left( \Gamma \left( A\right) ;\Gamma \left( \func{%
End}E\right) \right) $ determined by $\overline{\nabla }^{0}:\Gamma \left(
A\right) \rightarrow 
%TCIMACRO{\TeXButton{CDO}{\mathscr{CDO}}}%
%BeginExpansion
\mathscr{CDO}%
%EndExpansion
\left( \limfunc{End}E\right) $, $\overline{\nabla }_{a}^{0}=\left[ \nabla
_{a}^{0},\bullet \right] $ for all $a\in \Gamma \left( A\right) $, and $%
\left[ \lambda ,\lambda \right] \in \mathcal{A}lt_{\mathbb{R}}^{2}\left(
\Gamma \left( A\right) ;\Gamma \left( \func{End}E\right) \right) $ is given
by $\left[ \lambda ,\lambda \right] \left( a,b\right) =\left[ \lambda \left(
a\right) ,\lambda \left( b\right) \right] $ for all $a,b\in \Gamma \left(
A\right) $.

\begin{lemma}
\emph{\cite{Balcerzak}} For two $\mathbb{R}$-linear connections $\nabla ^{0}$%
,\thinspace $\nabla ^{1}:\Gamma \left( A\right) \rightarrow 
%TCIMACRO{\TeXButton{CDO}{\mathscr{CDO}}}%
%BeginExpansion
\mathscr{CDO}%
%EndExpansion
\left( E\right) $ the following properties hold: 
\begin{equation}
(R^{\nabla ^{\limfunc{aff}_{1}}})_{\frac{\partial }{\partial t},a}\left( \nu
\circ \limfunc{pr}\nolimits_{2}\right) _{|\left( t,\bullet \right) }=\lambda
\left( a\right) \left( \nu \right) ,  \label{properity_1}
\end{equation}%
\begin{equation}
(R^{\nabla ^{\limfunc{aff}_{1}}})_{a,b}\left( \nu \circ \limfunc{pr}%
\nolimits_{2}\right) _{|\left( t,\bullet \right) }=\left( 1-t\right) \cdot
R_{a,b}^{\nabla ^{0}}\left( \nu \right) +t\cdot R_{a,b}^{\nabla ^{1}}\left(
\nu \right) +\left( t^{2}-t\right) \cdot \left[ \lambda ,\lambda \right]
_{\left( a,b\right) }\left( \nu \right)  \label{properity_2}
\end{equation}%
for all $a,b\in \Gamma \left( A\right) $, $\nu \in \Gamma \left( E\right) $, 
$t\in \mathbb{R}$.
\end{lemma}

\paragraph{\textbf{The Chern-Simons forms of the first and the second rank}}

\ 

Let $\theta \in \mathcal{A}lt_{\mathbb{R}}^{1}\left( \Gamma \left( A\right)
;\Gamma \left( \func{End}E\right) \right) $, $\nabla :\Gamma \left( A\right)
\rightarrow 
%TCIMACRO{\TeXButton{CDO}{\mathscr{CDO}}}%
%BeginExpansion
\mathscr{CDO}%
%EndExpansion
\left( E\right) $ be an $\mathbb{R}$-linear connection of $A$ on $E$.
Therefore $\nabla +\theta $ is also an $\mathbb{R}$-linear $A$-connection on 
$E$, and $\limfunc{cs}\nolimits_{1}\left( \nabla ,\nabla +\theta \right) \in 
\mathcal{A}lt_{\mathbb{R}}^{1}\left( \Gamma \left( A\right) ;%
%TCIMACRO{\TeXButton{C}{\mathscr{C}}}%
%BeginExpansion
\mathscr{C}%
%EndExpansion
^{\infty }\left( M\right) \right) $ is given by $\limfunc{cs}%
\nolimits_{1}\left( \nabla ,\nabla +\theta \right) \left( a\right) =\limfunc{%
tr}\left( \theta \left( a\right) \right) $, $a\in \Gamma \left( A\right) $.
Moreover, we conclude from (\ref{properity_1}), (\ref{properity_2}) and (\ref%
{curvature_and_pair}) that%
\begin{equation*}
\limfunc{tr}\left( R^{\nabla ^{\limfunc{aff}_{1}}}\right) ^{2}\left( \frac{%
\partial }{\partial t},\bullet \right) _{|\left( t,\bullet \right) }=2%
\limfunc{tr}\left( \theta \wedge R^{\nabla ^{0}}+t\cdot \theta \wedge d_{%
\mathbb{R}}^{\overline{\nabla ^{0}}}\theta +t^{2}\cdot \theta \wedge \theta
\wedge \theta \right)
\end{equation*}%
for all $a_{1},a_{2},a_{3}\in \Gamma \left( A\right) $, $t\in \mathbb{R}$,
hence%
\begin{equation*}
\limfunc{cs}\nolimits_{2}\left( \nabla ,\nabla +\theta \right) =\limfunc{tr}%
\left( 2\theta \wedge R^{\nabla }+\theta \wedge d_{\mathbb{R}}^{\overline{%
\nabla }}\theta +\frac{2}{3}\theta \wedge \theta \wedge \theta \right) .
\end{equation*}%
If $\nabla $\ and $\nabla +\theta $ are both flat, then $d_{\mathbb{R}}^{%
\overline{\nabla }}\theta =-\theta \wedge \theta $, which then yields 
\begin{equation*}
\limfunc{cs}\nolimits_{2}\left( \nabla ,\nabla +\theta \right) =-~\frac{1}{3}%
\limfunc{tr}\left( \theta \wedge \theta \wedge \theta \right) .
\end{equation*}

For every manifold $M$ of an odd dimension $2m-1$, $\limfunc{cs}%
\nolimits_{m}\left( \nabla ,\nabla +\theta \right) $ is closed. In the case
where $M$ is a $3$-dimensional manifold, $\limfunc{cs}\nolimits_{2}\left(
\nabla ,\nabla +\theta \right) $\ is closed and is given by the above
formula; if additionally $\nabla $\ is flat, we see that%
\begin{equation}
\limfunc{cs}\nolimits_{2}\left( \nabla ,\nabla +\theta \right) =\limfunc{tr}%
\left( \theta \wedge d_{\mathbb{R}}^{\overline{\nabla }}\theta +\frac{2}{3}%
\theta \wedge \theta \wedge \theta \right) .  \label{formulaWZ}
\end{equation}%
(\ref{formulaWZ}) is a generalization of the known formula for tangent
bundles of smooth, compact, oriented, three dimensional manifolds and
standard connections (see for example \cite{Zhang}) to arbitrary rank three
vector bundles and $\mathbb{R}$-linear connections.

Moreover, we add (see \cite{Balcerzak}) that if both $\mathbb{R}$-linear
connections $\nabla ^{0}$,\thinspace $\nabla ^{1}:\Gamma \left( A\right)
\rightarrow 
%TCIMACRO{\TeXButton{CDO}{\mathscr{CDO}}}%
%BeginExpansion
\mathscr{CDO}%
%EndExpansion
\left( E\right) $ of a Lie algebroid $A$ on a vector bundle $E$ are flat,
then the Chern--Simons $\mathbb{R}$-linear form $\limfunc{cs}%
\nolimits_{p}\left( \nabla ^{0},\nabla ^{1}\right) $ is equal to $\left(
-1\right) ^{p+1}\,\frac{\,p!\left( p-1\right) !\,}{\left( 2p-1\right) !}%
\func{Tr}_{\ast }\left( \lambda ^{2p-1}\right) $. In particular, for any
flat $\mathbb{R}$-linear connection $\nabla $ of $A$ on $E$, $\nabla ^{h}$
is also flat and we conclude that the class $\limfunc{u}\nolimits_{2p-1}%
\left( A,E\right) $ is represented by the form 
\begin{equation*}
\left( -1\right) ^{p+1}\,\frac{\,p!\left( p-1\right) !\,}{\left( 2p-1\right)
!}\func{Tr}_{\ast }\left( \omega ^{2p-1}\right) ,
\end{equation*}%
where $\omega =\nabla ^{h}-\nabla \in \mathcal{A}lt_{\mathbb{R}}^{1}\left(
\Gamma \left( A\right) ;\Gamma \left( \func{End}E\right) \right) $.

\ 

\ \vspace{2cm}

{\small \hspace{-0.45cm}Bogdan Balcerzak}

{\small \hspace{-0.45cm}Institute of Mathematics}

{\small \hspace{-0.45cm}Technical University of \L \'{o}d\'{z}}

{\small \hspace{-0.45cm}W\'{o}lcza\'{n}ska 215}

{\small \hspace{-0.45cm}90-924 \L \'{o}d\'{z}, Poland}

{\small \hspace{-0.45cm}\emph{E-mail address}: bogdan.balcerzak@p.lodz.pl}

\end{document}